\documentclass[english,12pt]{amsart}
\usepackage[T1]{fontenc}
\usepackage[latin9]{inputenc}
\usepackage{amsthm}
\usepackage{amssymb}

\makeatletter
\numberwithin{equation}{section}
\numberwithin{figure}{section}
\theoremstyle{plain}
\newtheorem{thm}{\protect\theoremname}[section]
  \theoremstyle{definition}
  \newtheorem{defn}[thm]{\protect\definitionname}
  \theoremstyle{plain}
  \newtheorem{lem}[thm]{\protect\lemmaname}
  \theoremstyle{definition}
  \newtheorem{example}[thm]{\protect\examplename}
  \theoremstyle{plain}
  \newtheorem{cor}[thm]{\protect\corollaryname}

\usepackage{amscd}

\theoremstyle{plain}

\newcommand{\R}{\mathbb{R}}

\newcommand{\dx}{{\mathrm d}x}
\newcommand{\core}{{C_0^\infty}}

\makeatother

\usepackage{babel}
  \providecommand{\corollaryname}{Corollary}
  \providecommand{\definitionname}{Definition}
  \providecommand{\examplename}{Example}
  \providecommand{\lemmaname}{Lemma}
\providecommand{\theoremname}{Theorem}

\begin{document}

\title[Profile decompositions]{Trudinger-Moser inequality with remainder terms}

\author{Cyril Tintarev}

\address{Department of Mathematics, Uppsala University, P.O.Box 480, 75 106
Uppsala, Sweden }

\email{tintarev@math.uu.se}

\thanks{The research was in part supported by Vetenskapsrådet - Swedish Research
Council.}


\begin{abstract}
The paper gives the following improvement of the Trudinger-Moser inequality:

\begin{equation}
\sup_{\int_{\Omega}|\nabla u|^{2}\dx-\psi(u)\le1,\, u\in\core(\Omega)}\int_{\Omega}e^{4\pi u^{2}}\dx<\infty,\;\Omega\in\R^{2},\label{eq:TMmod}
\end{equation}
related to the Hardy-Sobolev-Mazya inequality in higher dimensions.
We show (\ref{eq:TMmod}) with $\psi(u)=\int_{\Omega}V(x)u^{2}\dx$
for a class of $V>0$ that includes 

\[
V(r)=\frac{1}{4r^{2}(\log\frac{1}{r})^{2}\max\lbrace\sqrt{\log\frac{1}{r}},1\rbrace}\;,
\]
which refines two previously known cases of (\ref{eq:TMmod}) proved
by Adimurthi and Druet \cite{AdiDruet} and by Wang and Ye \cite{WangYe}.
In addition, we verify (\ref{eq:TMmod}) for $\psi(u)=\lambda\|u\|_{p}^{2}$,
as well as give an analogous improvement for the Onofri-Beckner inequality.
\end{abstract}
\maketitle
\smallskip{}


\section{\label{intro}Introduction.}

The Trudinger-Moser inequality (\cite{Yud,Peetre,Poh,Trud,Moser})
\begin{equation}
\sup_{\int_{\Omega}|\nabla u|^{2}\dx\le1,\, u\in\core(\Omega)}\int_{\Omega}e^{4\pi u^{2}}\dx<\infty\;,\label{eq:TM}
\end{equation}
where $\Omega\subset\R^{2}$ is a bounded domain, is an analog of
the limiting Sobolev inequality in $\R^{N}$ with $N\ge3$: 

\begin{equation}
\sup_{\int_{\R^{N}}|\nabla u|^{2}\dx\le1,\, u\in\core(\R^{N})}\int|u|^{2^{*}}\dx\,<\infty,\;2^{*}=\frac{2N}{N-2}.\label{eq:Sob}
\end{equation}
\[
\]
We recall that restriction of inequalities involving the gradient
norm to bounded domains is of essense when $N=2$, since the completion
of $\core(\R^{2})$ in the gradient norm is not a function space,
and, moreover, since $\int_{B}|\nabla u|^{2}\dx$ on the unit disk
$B\subset\R^{2}$ coincides with the quadratic form of the Laplace-Beltrami
operator on the hyperbolic plane (a \emph{complete} \emph{non-compact}
Riemannian manifold) when expressed in the coordinates of the Poincaré
disk. 

Both limiting Trudinger-Moser and Sobolev inequalities are optimal
in the sense that they are false for any nonlinearity that grows as
$s\to\infty$ faster than $e^{4\pi s^{2}}$resp $s^{2^{*}}$. Inequality
(\ref{eq:Sob}) is also false if the nonlinearity $|u|^{2^{*}}$ is
multiplied by an unbounded radial monotone function, although (\ref{eq:TM})
on the unit disk holds also when the integrand is replaced by $\frac{e^{4\pi u^{2}}-1}{(1-r)^{2}}$
(\cite{AdiTi,MansanIneq}). $ $$ $

This paper studies another refinement of (\ref{eq:TM}), whose analogy
in the case $N\ge3$ is the Mazya\textquoteright{}s refinement of
(\ref{eq:Sob}), known as Hardy-Sobolev-Mazya inequality ({[}15{]}): 

\begin{equation}
\sup_{\int_{\R^{N}}|\nabla u|^{2}\dx-\int_{\R^{N}}V_{m}(x)u^{2}\dx\le1,\, u\in\core(\R^{N})}\int_{\R^{N}}|u|^{2^{*}}\dx\,<\infty,\label{eq:HSM}
\end{equation}

\[
\]
where 

\[
V_{m}(x)=\left(\frac{m-2}{2}\right)^{2}\frac{1}{|x_{1}+\dots+x_{m}|^{2}},\; m=1,...,N-1.
\]
It is false when $m=N$, and similarly, inequality (\ref{eq:TMmod})
does not hold with $\psi(u)=\int_{B}V(|x|)u^{2}\dx$, if $V$ is the
two-dimensional counterpart of the Hardy's radial potential, the Leray's
potential 

\[
V_{\mathrm{Leray}}(r)=\frac{1}{4r^{2}(\log\frac{1}{r})^{2}}.
\]
When $\psi(u)=\int_{\Omega}V(x)u^{2}\dx$ , inequality (\ref{eq:TMmod})
has been already established for two specific potentials $V$. In
one case, proved by Adimurthi and Druet \cite{AdiDruet}, $V(x)=\lambda<\lambda_{1}$,
and $\lambda_{1}$ is the first eigenvalue of the Dirichlet Laplacian
in $\Omega$. Note only that the inequality stated as a main result
in \cite{AdiDruet} is formally weaker, but it immediately implies
(\ref{eq:TMmod}) with $V(x)=\lambda<\lambda_{1}$ via an elementary
argument). It was conjectured by Adimurthi \cite{AdiPC} that the
inequality remains valid whenever one replaces $\int_{\Omega}\lambda u^{2}\dx$
with a general weakly continuous functional $\psi$, as long as $\|\nabla u\|_{2}^{2}-\psi(u)>0$
for $u\neq0$. Another known case of the inequality (\ref{eq:TMmod}),
with $\psi(u)=\int_{B}\frac{u^{2}}{(1-r^{2})^{2}}\dx$, is due to
Wang and Ye \cite{WangYe}. Note that the result of Wang and Ye involves
a non-compact remainder term, and that via conformal maps it extends
to general domains. 

In deciding about the natural counterpart of the Hardy-Sobolev-Mazya
inequality in the two-dimensional case, we have to make a choice,
which is insignificant in the case $N\ge3$, between using the functional
$\int e^{4\pi u^{2}}$ and the Orlitz norm $\|u\|_{\mathrm{Orl}}$
associated with the integrand (in terms of the standard definition,
with the function $e^{4\pi s^{2}}-1$). The difference between the
case $N\ge3$ and $N=2$ is in the fact that (\ref{eq:HSM}) can be
equivalently rewritten as 

\[
\int_{\R^{N}}|\nabla u|^{2}\dx-\int_{\R^{N}}V_{m}(x)u^{2}\dx\ge C\|u\|_{2^{*},}^{2}
\]
while from 

\begin{equation}
\int_{\Omega}|\nabla u|^{2}\dx-\psi(u)\ge C\|u\|_{\mathrm{Orl}}^{2}\label{eq:triv}
\end{equation}
for $N=2$ inequality (\ref{eq:TMmod}) does not follow, and insteaad
one has its weaker version, with the bound on $\int_{\Omega}e^{Cu^{2}}\dx$
with $some$ $C$. In particular, in the case of Adimurthi-Druet,
$V(x)=\lambda<\lambda_{1}$, inequality (\ref{eq:triv}) is completely
trivial while their actual result is very sharp. This example explains
why we, following Wang and Ye, treat (\ref{eq:TMmod}), and not (\ref{eq:triv}),
as a natural counterpart of (\ref{eq:HSM}).

The objective of this paper is to prove the inequality (\ref{eq:TMmod})
with the more general (and in particular, stronger) remainder term
$\psi(u)$ than in the two known cases. In Section 2 we study the
case $p=2$ and the radial potential on a unit disk, in Section 3
we extend the result to general bounded domains and to the values
$p>2$. In Section 4 we give corollaries to the inequalities, prove
a related refinement of Onofri-Beckner inequality, and list some open
problems. 

In what follows, $B$ will denote an open unit disk, $||\cdot\|_{p}$
will mean the $L^{p}(\Omega)$-norm when the domain is specified,
and the subspace of radial functions of, say, Sobolev space $H_{0}^{1}(B)$
will be denoted $H_{0,\mathrm{rad}}^{1}(B).$

\section{Remainder with a singular potential. }

\subsection*{Ground state alternative. }

We summarize first some relevant results on positive elliptic operators
with singular potentials, drawing upon \cite{pt-JFA}. 

Let $\Omega\subset\R^{N}$ be a domain, and let $V$ be a continuous
function in $\Omega$. We consider the functional 

\begin{equation}
Q_{V}(u)=\int_{\Omega}|\nabla u|^{2}\dx-\psi(u),\;\psi(u)=\int_{\Omega}V(x)u^{2}\dx,\; u\in\core(\Omega).\label{eq:QV}
\end{equation}
Assuming that $Q_{V}\ge0$, one says that $\varphi\neq0$ is a \emph{ground
state} of the quadratic form $Q_{V}$ if there exists a sequence $u_{k}\in\core(\Omega)$,
convergent to $\varphi$ in $H_{\mathrm{loc}}^{1}(\Omega)$, such
that $Q_{V}(u_{k})\to0$. Ground states are sign definite and, up
to a constant multiple, unique in the class of positive solutions
(that is, positive solutions without global integrability requirements
or boundary conditions). If, additionally, $\varphi\in H_{0}^{1}(\Omega)$,
then $\varphi$ is a minimizer for the Rayleigh quotient 

\[
\inf_{u\in H_{0}^{1}(\Omega),u\neq0}\frac{\|\nabla u\|_{2}^{2}}{\int_{\Omega}V(x)u^{2}\dx}\,.
\]
There are ground states, however, for which $\|\nabla\varphi\|_{2}=\infty$.
This is the case, in particular, for the ground state $\varphi(x)=\sqrt{\log\frac{1}{|x|}}$
in the case of Leray potential 

\[
Q_{V}=\int_{B}|\nabla u|^{2}\dx-\int_{B}V_{\mathrm{Leray}}u^{2}\dx.
\]
(Leray inequality, \cite{Leray}, states that this form is nonnegative.)
Similarly, Hardy inequality in $\R^{N}$, $N\ge3$, with the radial
potential $V_{N}$ admits a ground state $\varphi(x)=|x|^{\frac{2-N}{2}}$,
whose gradient norm is infinite as well. 

Existence of a ground state is connected to the property of weak coercicity.
The form (\ref{eq:QV}) is called weakly coersive if there exists
an open set $E$ relatively compact in $\Omega$ and a constant $\delta>0$,
such that 

\[
Q_{V}(u)\ge\delta\left(\int_{E}u\dx\right)^{2},\; u\in\core(\Omega).
\]
An equivalent criterion of weak coercivity (see \cite{TakTin}) is
a seemingly stronger condition that there exists a continuous function
W > 0 such that 

\[
Q_{V}(u)\ge\int_{\Omega}W(x)(|\nabla u|^{2}+u{}^{2})\dx,\; u\in\core(\Omega).
\]
It is well known that the form (\ref{eq:QV}) is nonnegative if and
only if it admits a positive solution. However, not any positive solution
is a ground state, and in fact, existence of a ground state and weak
coercivity for a nonnegative form are mutually exclusive. 
\begin{thm}
\textbf{(Ground state alternative of Murata {[}17, 20{]})} A nonnegative
functional (\ref{eq:QV}) is either weakly coercive or has a ground
state. 
\end{thm}
If the form (\ref{eq:QV}) is nonnegative (and thus admits a positive
solution $v$) it can be represented as an integral of a positive
function. This representation is known as \emph{ground state transform}
or \emph{Jacobi identity}:

\[
\int_{\Omega}|\nabla u|^{2}\dx-\int_{\Omega}V(x)u^{2}\dx=\int_{\Omega}v^{2}|\nabla\frac{u}{v}|^{2}\dx.
\]

\subsection*{Remainder in the Trudinger-Moser inequality, radial case. }
\begin{defn}
We say that a radial function on the unit disk $V(|x|)\in\mathcal{V}$
if $V(r)$ is a nonnegative continuous function on $(0,1)$ and the
function $r\mapsto(1-r^{2})^{2}V(r)$ is nonincreasing.\end{defn}
\begin{lem}
\label{lem:radred}If $V\in\mathcal{V}$ then 

\begin{equation}
\sup_{u\in H_{0}^{1}(B),\; Q_{V}(u)\le1}\int_{B}e^{4\pi u^{2}}\dx=\sup_{u\in H_{0,\mathrm{rad}}^{1}(B),\; Q_{V}(u)\le1}2\pi\int_{B}e^{4\pi u(r)^{2}}r\mathrm{d}r.\label{eq:radred}
\end{equation}
\end{lem}
\begin{proof}
Consider $B$ as the Poincaré disk representing the hyperbolic plane
$\mathbb{H}^{2}$. The quadratic form of Laplace-Beltrami operator
on $\mathbb{H}^{2}$ in the Poincaré disk coordinates is $\int_{B}|\nabla u|^{2}\dx$.
Let $u^{\#}$ denote the spherical decreasing rearrangement of $u\in H_{0}^{1}(B)$
relative to the Riemannian measure of the Poincaré disk, $d\mu=\frac{4dx}{(1-r^{2})^{2}}$,
and recall that the Hardy-Littlewood and the Polia-Szegö inequalities
relative to these rearrangements remain valid ({[}5{]}). In particular,
by the Hardy-Littlewood inequality,

\[
\int_{B}V(|x|)u(x)^{2}\dx=\int_{B}\frac{1}{4}(1-|x|^{2})^{2}V(|x|)u(x)^{2}\mathrm{d}\mu
\]

\[
\le\int_{B}\frac{1}{4}(1-r^{2})^{2}V(r)u^{\#}(r){}^{2}\mathrm{d}\mu=\int_{B}V(r)u^{\#}(r)^{2}\mathrm{d}x,
\]
and thus, taking into account the Polia-Szegö inequality, we have
$Q_{V}(u)\ge Q_{V}(u^{\#})$. From this and the ``hyperbolic'' Hardy-Littlewood
inequality applied to $\int e^{4\pi u^{2}}\dx$ it follows that the
right hand side in (\ref{eq:radred}) is not less then the left hand
side, while the converse is trivial.\end{proof}
\begin{thm}
\label{thm:main}Let $N=2$, let $V\in\mathcal{V}$, and assume that,
for some $\alpha>0$, 

\begin{equation}
\lim_{r\to0}r^{2}(\log\frac{1}{2})^{2+\alpha}V(r)=0.\label{eq:Kato}
\end{equation}

Then the quantity 

\[
S_{V}=\sup_{u\in H_{0}^{1}(B),\; Q_{V}(u)\le1}J(u),\; J(u)=\int_{B}e^{4\pi u^{2}}\dx,
\]

is finite if and only if the quadratic form $Q_{V}$ is weakly coercive. \end{thm}
\begin{proof}
\emph{1. Necessity}. Assume that $Q_{V}$ is not weakly coercive.
If $Q_{V}(w)<0$ for some $w\in H_{0}^{1}(B)$, then $J(kw)\to\infty$
and thus $S_{V}=+\infty$. Assume now that $Q_{V}\ge0$. Then by the
ground state alternative, $Q_{V}$ has a ground state $\varphi>0$
approximated by a $\core$-sequence $u_{k}\to\varphi$ in $H_{\mathrm{loc}}^{1}(B)$
such that $Q_{V}(u_{k})\to0$. Then, noting that there exist $\epsilon>0$
and $\delta>0$, such that for each $k$, inequality $u_{k}\ge\epsilon$
holds on some set of measure larger than $\delta$, we have $J(u_{k}/\sqrt{Q_{V}(u_{k})})\to\infty$,
which again yields $S_{V}=+\infty$. (Of course, $Q_{V}(u_{k})\neq0$
since otherwise $u_{k}$ equals $\varphi$ up to a constant multiple,
which is a contrasiction since $\varphi>0$ and $u_{k}\in\core(B)$.) 

\emph{2. Sufficiency.} Assume that $Q_{V}$ is weakly coercive. By
Lemma \ref{lem:radred} it suffices to consider the problem restricted
to radial decreasing functions. Since $Q_{V}$ is nonnegative, equation
$Q_{V}^{'}(u)=0$ has a positive radial $C^{1}-$solution $\varphi$.
The latter fact can be infered from the fact that V , by (\ref{eq:Kato}),
belongs to the local Kato class $\mathcal{K}_{2}$ (see \cite{AizSim}).
Let us normalize $\varphi$ by dividing it by $\varphi(0)$, so that
$\varphi(0)=1$ and $\varphi(r)\le1$. Define now now 

\begin{equation}
s(r)=e^{\int_{1/e}^{r}\frac{\mathrm{d}t}{t\varphi(t)^{2}}}\;,\;0<r<1,\label{eq:s(r)}
\end{equation}
so that the function $s(r)$ satisfies

\[
\frac{s'(r)}{s(r)}=\frac{1}{r\varphi(r)^{2}}\,.
\]
Since $\varphi(0)=1$, we have $s(r)=\gamma r+o_{r\to0}(r)$ with
some $\gamma>0$, which implies that $s(r)$ defines a monotone $C^{1}$-homeomorphism
between $[0,1)$ and $[0,s(1))$, where $s(1)=\lim_{r\to1}s(r)$ may
be, generally speaking, infinite. Let $w:[0,s(1))\to[0,1)$ be the
function

\begin{equation}
w(s(r))=u(r)/\varphi(r)\label{eq:w(s(r))}
\end{equation}
Then, writing $Q_{V}$ in the ground state transform form and changing
the radial integration variable from $r$ to $s(r)$ we get 

\[
Q_{V}(u)=\int_{B_{(s(1))}}|w'(|x|)^{2}\dx.
\]
Assume first that $s(1)<\infty$. Then, taking into account that $\varphi\le1$
and $r\le s(r)/s(1)$ (which is easy to infer from (\ref{eq:s(r)})),
we have 

\[
S_{V}\le\sup_{\int_{B_{s(1)}}|\nabla w|^{2}=1}\int_{B_{s(1)}}e^{4\pi\varphi(r(s))^{2}w(s)^{2}}s\mathrm{d}s\mathrm{d}\theta\le\sup_{\int_{B_{s(1)}}|\nabla w|^{2}=1}\int_{B_{s(1)}}e^{4\pi w^{2}}\mathrm{d}x<\infty,
\]
which proves the theorem in this case. Assume now that $s(1)=+\infty$.
Then $Q_{V}(u)=\int_{\R^{2}}|\nabla w|^{2}\dx$. Let $w_{k}(s)=1$
for $r<k$, $w_{k}(s)=\frac{log\frac{k^{2}}{s}}{k}$ for $k\le s<k^{2}$,
$w_{k}(s)=0$ for $s\ge k^{2}$. Then the sequence $\varphi(r)w_{k}(s(r))$
fulfills the definition of approximating sequence for the ground state
$\varphi$ of $Q_{V}$. This, however, in view of the ground state
alternative, contradicts the assumption that $Q_{V}$ is weakly coersive.
Thus $s(1)<\infty$ , in which case the theorem is already proved. \end{proof}
\begin{example}
(a) Adimurthi and Druet, \cite{AdiDruet}: the constant potential
$V(r)=\lambda<\lambda_{1}$; where $\lambda_{1}$ is the first eigenvalue
of the Dirichlet Laplacian, satisfies the assumptions of Theorem \ref{thm:main}.

(b) Potential $V_{\mathrm{Leray}}(r)=\frac{1}{4r^{2}(log\frac{1}{r})^{2}}$
gives $S_{V}=+\infty$, since $Q_{V_{\mathrm{Leray}}}$ has a ground
state $\varphi(r)=\sqrt{\log\frac{1}{r}}$.

(c) Another potential satisfying the assumptions of Theorem \ref{thm:main}
is 

\[
V_{\gamma}(r)=\frac{1}{4r^{2}(\log\frac{1}{r})^{2}\max\lbrace(\log\frac{1}{r})^{\gamma},1\rbrace},\gamma\in(0,\frac{4}{e^{2-1}}).
\]
Since $V_{\gamma}<V_{\mathrm{Leray}}$ with the strict inequality
on $(0,e^{-1})$, $Q_{V}$ is weakly coercive. The potential $V(r)=\frac{1}{(1-r^{2})^{2}}$,
for which inequality (\ref{eq:TMmod}) was proved in \cite{WangYe}$,$
is smaller than $V_{\gamma}(r)$, which (or comparison with the Hardy
inequality) implies that $V_{\gamma}(r)$ has the optimal multiplicative
constant and that the set $\lbrace Q_{V_{\gamma}}(u)\le1\rbrace$
is not bounded in $H_{0}^{1}(B)$. 
\end{example}

\section{The non-radial case and the $L^{p}$ - remainder. }

We start with an elementary extention of the result of the previous
section to the general bounded domain. We recall that $w^{\#}$ denotes
rearrangement with respect to the Riemannian measure on the hyperbolic
plane.
\begin{thm}
\label{thm:3.1.nonrad}Let $\Omega\subset\R^{2}$ be a bounded domain,
$R=\sqrt{\frac{|\Omega|}{\pi}}$, $V\in L^{1}\mathrm{_{loc}}(\Omega)$,
and let 

\[
\tilde{V}(r)=\frac{[(1-|x|^{2}/R^{2})^{2}V(\frac{x}{R})]^{\#}(r)}{(1-r^{2})^{2}}.
\]

\begin{thm}
Assume that $\tilde{V}\in\mathcal{V}$ and satisfies (\ref{eq:Kato}),
with some $\alpha>0$. If the form $Q_{\tilde{V}}:H_{0,\mathrm{rad}}^{1}(B)\to\R$,
defined as in (\ref{eq:QV}), is weakly coersive, then
\end{thm}
\[
S_{V}=\sup_{u\in\core(\Omega):Q_{V}(u)\le1}\int_{\Omega}e^{4\pi u^{2}}dx<\infty.
\]
\end{thm}
\begin{proof}
Rescale the problem to a domain of the area $\pi$. Reduce the problem
to the radial problem on a unit disk by using rearrangements with
respect to the Riemannian measure of $\mathbb{H}^{2}$ and apply Theorem
\ref{thm:main}. 
\end{proof}
For the rest of the section we consider the maximization problem 

\[
S_{\lambda,p}=\sup_{u\in\core(\Omega):Q_{\lambda,p}(u)\le1}\int_{\Omega}e^{4\pi u^{2}}dx<\infty,
\]
where 

\[
Q_{\lambda,p}(u)=\int_{\Omega}|\nabla u|^{2}\dx-\lambda\|u\|_{p}^{2},
\]
and $\Omega\subset\R^{2}$. We will use the following constant: 

\[
\lambda_{p}=\inf_{u\in C_{0}^{1}(\Omega^{\star}):\,\|u\|_{p}=1}\int_{\Omega^{\star}}|\nabla u|^{2}\dx,\; p>0,
\]
where $\Omega^{\star}$ is the open ball of radius $\sqrt{\frac{|\Omega|}{\pi}}$
. 
\begin{thm}
\label{thm:p}Let $2<p<\infty$ and $\lambda<\lambda_{p}$. Then 

\[
S_{\lambda,p}=\sup_{u\in\core(\Omega):\, Q_{\lambda,p}(u)\le1}\int_{\Omega}e^{4\pi u^{2}}\dx<\infty.
\]
\end{thm}
\begin{proof}
It suffices to verify the assertion in restriction to positive radial
decreasing $H_{0}^{1}$-functions on $\Omega^{\star}$ when $\Omega^{\star}$
is the unit disk $B$. Let us represent $Q_{\lambda,p}(u)$ as $Q_{V_{u}}(u)$
with $V_{u}(u)=\lambda\frac{u^{p-2}}{\|u\|_{p}^{p-2}}$ , $u\in H_{0,\mathrm{rad}}^{1}$.
Observe that by Hölder inequality 
\[
\int_{B}u^{p-2}\varphi^{2}\dx\le\|u\|_{p}^{p-2}\|\varphi\|_{p}^{2},
\]
 and therefore $Q_{V_{u}}(\varphi)\ge Q_{\lambda,p}(\varphi)\ge0$.
Consequently, there exists a positive radial solution $\varphi_{u}$
to the linear equation$-\Delta\varphi=V_{u}\varphi$ in $B$. Since,
by the standard radial estimate, $V_{u}(r)\le C(\log\frac{1}{r})^{\frac{p-2}{p}}$,
one has $\varphi_{u}\in C^{1}(B)$, and the maximum of $\varphi_{u}$
is at the origin. We assume without loss of generality that $\varphi_{u}(0)=1$.
By the ground state transform we have for any $v\in\core(B)$,

\[
Q_{V_{u}}(v)=\int_{B}\varphi_{u}^{2}|\nabla\frac{v}{\varphi_{u}}|^{2}\dx,\, v\in\core(B).
\]
Let now 

\[
s_{u}(r)=e^{\int_{e^{-1}}^{r}\frac{dt}{t\varphi_{u}(t)^{2}}},\;0<r<1,
\]
and note that this function satisfies 

\[
\frac{s'_{u}(r)}{s_{u}(r)}=\frac{1}{r\varphi_{u}(r)^{2}}.
\]
Observe that since $\varphi_{u}(0)=1$ and $\varphi_{u}$ is a classical
solution, we have $s_{u}(r)=\gamma r+o_{r\to0}(r)$ with some $\gamma>0$,
and thus the mapping $r\mapsto s_{u}(r)$ is a monotone $C^{1}$-homeomorphism
between {[}0, 1) and $[0,s_{u}(1))$. We will show now that $\varphi_{u}$
is bounded away from zero near $r=1$, uniformly in a $H_{0,\mathrm{rad}}^{1}(B)$-ball
of $u$. First note that if for some $u\in H_{0,\mathrm{rad}}^{1}(B)$
one has $\varphi_{u}(1)=0$, then $\varphi_{u}$ is the first eigenfunction
for the Dirichlet eigenvalue problem $-\Delta\varphi=V_{u}\varphi$
in $B$. From the Hölder inequality and the definition of $\lambda_{p}$$ $
we get: 

\[
\int_{B}|\nabla\varphi_{u}|^{2}\dx=\int_{B}V_{u}\varphi^{2}\dx\le\lambda\left(\int_{B}\left(\frac{u}{\|u\|_{p}}\right)^{p}\right)^{1-2/p}\left(\int_{B}\varphi_{u}^{p}\right)^{2/p}
\]

\[
\le\lambda\lambda_{p}^{-1}\int_{B}|\nabla\varphi_{u}|^{2}\dx<\int_{B}|\nabla\varphi_{u}|^{2}\dx,
\]
a contradiction. Thus $\varphi_{u}(1)>0$ for any $u$, and it remains
to show that $\varphi_{u}(r)$ has a common positive lower bound for
all $u$ and all $r$ near $1$. Indeed, assume that there is a sequence
$u_{k}$ with $Q_{\lambda,p}(u_{k})\le1$, and a sequence $r_{k}\to1$
such that $\varphi_{u_{k}}(r_{k})\to0$ and $-\Delta\varphi_{u_{k}}=\lambda u_{k}^{p-2}\varphi_{u_{k}}$.
Note that since $\lambda<\lambda_{p}$, the sequence $u_{k}$ is bounded
in $H_{0}^{1}(B)$, and without loss of generality we may assume that
$u_{k}\rightharpoonup u$ in $H_{0}^{1}(B)$ with $Q_{\lambda,p}(u)\le1$.
From here one can easily derive that $\varphi_{u_{k}}$ converges
uniformly to some nonnegative $\varphi$ with $\varphi(1)=0$, and
that $\varphi$ satisfies the equation $-\Delta\varphi=V_{u}\varphi$.
In other words, $\varphi=\varphi_{u}$ and we have $\varphi_{u}(1)=0$,
which is a contradiction. We conclude that there exists $\epsilon>0$
and $\delta>0$, such that $\inf_{r\in[1-\epsilon,1],\, u:Q_{\lambda,p}(u)\le1}\varphi_{u}(r)\ge\delta$.
This implies that there is a number S such that $s_{u}(1)\le S$ for
all $u$ satisfying $Q_{\lambda,p}(u)\le1$. 

For each $v\in H_{0,\mathrm{rad}}^{1}(B)$ define the following function
on $[0,s_{u}(1))$: 

\[
w_{v;u}(s_{u}(r))=v(r).
\]
Then, applying the ground state transform and the changing the radial
integration variable from $r$ to $s_{u}$, we have 

\[
Q_{V_{u}}(v)=\int_{B}\varphi_{u}^{2}|\nabla\frac{v}{\varphi_{u}}|^{2}\dx=\int_{B_{s_{u}(1)}}|w'_{v;u}(|x|)|^{2}\dx,\; v\in H_{0,\mathrm{rad}}^{1}(B).
\]

By setting $v=u$, we get from here 

\[
Q_{\lambda,p}(u)=\int_{B_{s_{u}(1)}}|w'_{u;u}(|x|)|^{2}\dx,\; v\in H_{0,\mathrm{rad}}^{1}(B).
\]

Then, taking into account that $\varphi_{u}\le1$ for every $u$,
we arrive at 

\[
S_{\lambda,p}\le S^{2}\sup_{\int_{B}|\nabla w|^{2}=1}\int_{B}e^{4\pi w(|x|)^{2}}\dx<\infty.
\]

which proves the theorem. 
\end{proof}

\section{Related inequalities }

The arguments in Sections 2 and 3 allow to give the following refinement
of the Onofri-Beckner inequality (\cite{Onofri,Beckner}). The original
inequality for the unit disk is

\begin{equation}
\log\left(\frac{1}{\pi}\int_{B}e^{u}\dx\right)+\left(\frac{1}{\pi}\int_{B}e^{u}\dx\right)^{-1}\le1+\frac{1}{16\pi}\int_{B}|\nabla u|^{2}\dx,\; u\in\core(B).\label{eq:flatOnofri}
\end{equation}

\begin{thm}
Let $\Omega=B$ and assume that $\psi(u)=\int_{B}Vu^{2}\dx$ with
$V$ as in Theorems \ref{thm:main} and \ref{thm:3.1.nonrad}, or
that $\psi(u)=\lambda\|u\|_{p}^{2}$, $\lambda<\lambda_{p}$, $p>2$,
as in Theorem \ref{thm:p}. Then for every $\; u\in\core(B)$,
\end{thm}
\begin{equation}
\log\left(\frac{1}{\pi}\int_{B}e^{u}\dx\right)+\left(\frac{1}{\pi}\int_{B}e^{u}\dx\right)^{-1}\le1+\frac{1}{16\pi}\left(\int_{B}|\nabla u|^{2}\dx-\psi(u)\right).\label{eq:NewOnofri}
\end{equation}

\begin{proof}
We give the proof for the case of the remainder term $\psi$ as in
Theorem \ref{thm:main}. The proofs in other cases are analogous.
By the standard rearrangement argument it suffices to consider the
radially symmetric functions. 

Assume firtst that $u\ge0$. Without loss of generality we may assume
that $u$ is radial. Let us use the coordinate transformation (\ref{eq:s(r)})
and the substitution (\ref{eq:w(s(r))}). Taking into account that
the function $F(t):=\log t+t^{-1}$ is increasing on $(1,\infty)$,
that the function $\varphi$, involved in the transformation, does
not exceed $1$, and that, as it is immediate from (\ref{eq:s(r)}),
$s(r)/s(1)\ge r$ we have from (\ref{eq:flatOnofri}) 

\[
F\left(\frac{1}{\pi s(1)^{2}}\int_{B_{s(1)}}e^{\varphi(r(s))w(s)}\frac{r(s)^{2}\varphi(r(s))^{2}}{s^{2}}\dx(s)\right)\le
\]

\[
\le1+\frac{1}{16\pi}\int_{B_{s(1)}}|\nabla w|^{2}\dx,\; w\in H_{0,rad}^{1}(B_{s(1)}).
\]
Using (\ref{eq:w(s(r))}) in order to return to the original variable
$u$, we immediately have (\ref{eq:NewOnofri}) for $u\ge0$. 

Consider now the case $u\le0$. Without loss of generality we again
assume that $u$ is radial. Then, taking into account (\ref{eq:s(r)}),
(\ref{eq:w(s(r))}), $\varphi\le1$, $r\le s(r)$, and the fact that
the function $F$ is decreasing on $(0,1)$, we have

\[
F\left(\frac{1}{\pi}\int_{B}e^{u}\dx\right)\le F\left(\frac{1}{\pi}\int_{B}e^{w(s(r))}\dx\right)
\]

\[
=F\left(\frac{1}{\pi s(1)^{2}}\int_{B_{s(1)}}e^{w(s)}\frac{s^{2}}{r(s)^{2}\varphi(r(s))^{2}}\dx(s)\right)
\]

\[
\le F\left(\frac{1}{\pi s(1)^{2}}\int_{B_{s(1)}}e^{w(s)}\dx(s)\right)
\]

\[
\le1+\frac{1}{16\pi}\int_{B_{s(1)}}|\nabla w|^{2}\dx=1+\frac{1}{16\pi}Q_{V}(u).
\]

Finally, we write a general $u$ as $u=u^{+}+(-u^{-})$ and note that
the function $\log t+1/t$ is subadditive on $(0,\infty)$.We leave
it to the reader to prove the subadditivity with help of the following
sketch: collect the logarithmic terms in the subadditivity inequality
into a single logarithm, invert the logarithm, and replace the resulting
exponential function by its Taylor polynomial up to the order 2. Inequality
(\ref{eq:NewOnofri}) is then immediate from the cases where $u\ge0$
and $u\le0$. \end{proof}
\begin{cor}
\textbf{(Inequality of Adimurthi-Druet type.)} Let $Q(u)=\|\nabla u\|_{2}^{2}-\psi(u)$
be any of the functionals $Q_{V}$ as in Theorems \ref{thm:main}
and \ref{thm:3.1.nonrad}, or the functional $Q_{p}$, as in Theorem
\ref{thm:p}. Then 

\[
\sup_{\|\nabla u\|_{2}\le1}\int_{\Omega}e^{4\pi(1+\psi(u))u^{2}}\dx\le\sup_{\|\nabla u\|_{2}\le1}\int_{\Omega}e^{\frac{4\pi u^{2}}{1-\psi(u)}}\dx<\infty
\]
\end{cor}
\begin{proof}
Note first that the integral in the left hand side is smaller than
the integral in the right hand side by the inequality $(1+\psi)(1-\psi)<1$.
Let $u=\sqrt{\gamma}v$ with $\|\nabla v\|_{2}=1$. Then $Q(u)\le1$
is equivalent to $\gamma-\gamma\psi(v)\le1$, i.e. $\gamma\le\frac{1}{1-\psi(v)}$.
Write (\ref{eq:TMmod}), substitue $u^{2}=\gamma v^{2}$ into the
integral and rename $v$ as $u$. \end{proof}
\begin{cor}
Let $\|\cdot\|_{\mathrm{Orl}}$ denote the Orlicz norm associated
with the Trudinger-Moser functional on a bounded domain $\Omega\subset\R^{2}$,
and let $Q(u)=\|\nabla u\|_{2}^{2}-\psi(u)$ be any of the functionals
$Q_{V}$ as in Theorems \ref{thm:main} and \ref{thm:3.1.nonrad},
or the functional $Q_{p}$, as in Theorem \ref{thm:p} Then there
exists a $C>0$ such that 

\[
\int_{\Omega}|\nabla u|^{2}\dx-\psi(u)\ge C\|u\|_{\mathrm{Orl}}^{2}
\]

\end{cor}
Proof. Assume first that $Q(u)=1$. From the uniform bound on $\int_{\Omega}(e^{4\pi u^{2}}-1)\dx$
in (\ref{eq:TMmod}) follows a uniform bound for the Orlicz norm,
which yileds the inequality under the constraint $Q(u)=1$. It remains
to use the standard homogeneity argument.

\subsection*{Open problems. }
\begin{enumerate}
\item Does the inequality (\ref{eq:TMmod}) hold for general bounded $\Omega$,
all potentials $V$ of the local Kato class $\mathcal{K}_{2}$ and
all $p\in(0,\infty)$, as long as the constraint functional $Q$ remains
weakly coersive? 
\item When $\Omega=\mathbb{R}^{2}$, inequality (\ref{eq:TMmod}) with $Q(u)=\|\nabla u\|_{2}^{2}$
is false, since the form $\|\nabla u\|_{2}^{2}$ on the whole $\mathbb{R}^{2}$
admits a ground state 1. On the other hand, the inequality holds when
$Q(u)=\|\nabla u\|_{2}^{2}+\|u\|_{2}^{2}$ (Ruf, \cite{Ruf}). Furthermore,
as it is shown in \cite{MansanIneq}, inequality (\ref{eq:TMmod})
with $Q(u)=\|\nabla u\|_{2}^{2}$ holds for a simply connected (generally
unbounded) domain $\Omega\subset\R^{2}$ if and only if $\|\nabla u\|_{2}^{2}\ge\lambda\|u\|_{2}^{2}$
with some$\lambda>0$. In both results the condition is $L^{2}$-
coercivity, $Q(u)\ge C\|u\|_{2}^{2}$. It is natural then to ask,
for unbounded domains, if there are weaker coercivity conditions on
$Q$ that yield (\ref{eq:TMmod})? 
\item Since Hardy-Sobolev-Maz\textquoteright{}ya inequalities can be derived
from Caffarelli-Kohn-Nirenberg inequalities (\cite{CKN}) via the
ground state transform, it is natural to ask what could be an analog
of Caffarelli-Kohn-Nirenberg inequalities related to the remainder
estimates of the Hardy-Moser-Trudinger type.
\item Our reduction to the radial case is of tentative character, as it
is based on rearrangements specific to the hyperbolic plane which
resulted in a restrictive coniditon of weighted monotonicity on the
potential. Perhaps more general rearrangements satisfying Polia-Szegö
and Hardy-Littlewood inequalities (see \cite{MartinMilman}) can be
used to relax the monotonicity condition on the potential.\end{enumerate}

\end{document}